\documentclass[12pt]{article}
\usepackage{latexsym, amsthm, amsmath}
\usepackage{epsfig, xspace}
\usepackage{ifthen}
\usepackage[charter]{mathdesign}
\usepackage[svgnames]{xcolor}

\newcommand{\CC}{\mathbb{C}}

\newcommand{\ZZ}{\mathbb{Z}}

\newcommand{\RR}{\mathbb{R}}

\newcommand{\calH}{\mathscr{H}}
\newcommand{\calB}{\mathscr{B}}

\newcommand{\GL}{\mathrm{GL}}
\newcommand{\GLTNN}{\mathrm{GL_{\geq 0}}}
\newcommand{\Gr}{\mathrm{Gr}}
\newcommand{\GrTP}{\mathrm{Gr_{>0}}}
\newcommand{\GrTNN}{\mathrm{Gr_{\geq 0}}}
\newcommand{\CMA}{\CC^\mathrm{MA}}
\newcommand{\RMA}{\RR^\mathrm{MA}}
\newcommand{\Mat}{\mathrm{Mat}}
\newcommand{\MatTP}{\mathrm{Mat}_{>0}}
\newcommand{\MatTNN}{\mathrm{Mat}_{\geq 0}}

\newcommand{\boldu}{{\bf u}}
\newcommand{\boldw}{{\bf w}}
\newcommand{\boldx}{{\bf x}}
\newcommand{\boldy}{{\bf y}}
\newcommand{\boldz}{{\bf z}}

\newcommand{\Plucker}{Pl\"ucker\xspace}
\newcommand{\Branden}{Br\"and\'en\xspace}

\newcommand{\defn}[1]{\emph{#1}}

\newtheorem{lemma}{Lemma}[section]
\newtheorem{theorem}[lemma]{Theorem}

\newtheorem{proposition}[lemma]{Proposition}

\theoremstyle{definition}
\newtheorem{example}[lemma]{Example}

\newtheorem{remark}[lemma]{Remark}

\newenvironment{packedenum}{
\begin{enumerate}
  \setlength{\itemsep}{.25ex}
}{\end{enumerate}}

\newenvironment{packeditemize}{
\begin{itemize}
  \setlength{\itemsep}{.25ex}
}{\end{itemize}}

\title{Total nonnegativity and stable polynomials%
{%Removing the footnote symbol. 
    \footnotetext{% 
    \emph{2010 Mathematics Subject Classification.}
      Primary: 32A60; Secondary: 14M15, 14P10, 15B48
    }% 
}}

\author{Kevin Purbhoo}

\begin{document}

\maketitle

%%%%%%%%%%%%%%%%%%%%%%%%%%%%%%%%%%%%%%%%%%%%%%%%%%

\begin{abstract}
We consider homogeneous multiaffine polynomials whose coefficients
are the \Plucker coordinates of a point $V$ of the Grassmannian.
We show that such a polynomial is stable (with respect to the
upper half plane) if and only if $V$ is in the totally nonnegative
part of the Grassmannian.  To prove this, we consider an action of
matrices on multiaffine polynomials.  We show that 
a matrix $A$ preserves stability of polynomials if and only if
$A$ is totally nonnegative.  The proofs are applications of classical
theory of totally nonnegative matrices, and the generalized 
P\'olya--Schur theory of Borcea and \Branden.
\end{abstract}

%%%%%%%%%%%%%%%%%%%%%%%%%%%%%%%%%%%%%%%%%%%%%%%%%%

\section{Introduction}

A multivariate polynomial $f(\boldx) = f(x_1, \dots, x_n) \in \CC[\boldx]$ 
is said to be \defn{stable} 
if either $f \equiv 0$, or $f(\boldu) \neq 0$, for all 
$\boldu = (u_1, \dots, u_n) \in \calH^n$, where 
$\calH = \{u \in \CC \mid \mathrm{Im}(u) > 0\}$ 
denotes the upper half plane in $\CC$. 
The theory of stable polynomials generalizes and vastly extends
the theory of univariate real polynomials with only real roots.
Although the idea of considering polynomials (and more generally
analytic functions) with no zeros inside a domain has an extensive
history in complex analysis, more recent developments --- 
notably the generalized P\'olya--Schur theory of Borcea and \Branden
\cite{BB1, BB2} --- 
have generated new interest in the subject, and a wide variety of 
new applications have been discovered in areas
such as matrix theory, statistical mechanics, and combinatorics.
We refer the reader to the survey \cite{Wagner-survey} 
for an introduction
to the theory of stable polynomials and an overview of some of its 
applications.

Central to the theory %of stable polynomials 
is the vector space $\CMA[\boldx]$ of
\defn{multiaffine} polynomials. These are 
the polynomials in $\CC[\boldx]$ that have degree at most one in each 
individual variable.  The Grace--Walsh--Szeg\"o
coincindence theorem \cite{Grace, Szego, Walsh}
allows one to reduce
many problems about stable polynomials to the multiaffine
case; moreover, a number of applications of the theory, notably 
those involving matroid theory \cite{Branden, COSW}, statistical
mechanics \cite{BBL},
and the present paper, involve only multiaffine polynomials.

The \defn{Grassmannian} $\Gr(k,n)$ is the space of all 
$k$-dimensional linear subspaces of $\CC^n$.
There are two common ways to specify a point $V \in \Gr(k,n)$.  
The simplest is 
as the column space of a rank $k$ complex 
matrix $M \in \Mat(n \times k)$;
however, for any given $V$, this matrix $M$ is not unique.
A more canonical way to specify $V$ is via its \Plucker coordinates.
Let $M[I]$ denote the $k \times k$ submatrix of $M$ with row
set $I \in {[n] \choose k}$.
The \defn{Pl\"ucker coordinates} of $V$
are the maximal minors 
$\big[\det(M[I]) : I \in {[n] \choose k}\big]$.  
These are homogeneous coordinates for $V$, i.e. they are well-defined
up rescaling by a nonzero constant.
We can encode the Pl\"ucker coordinates of $V$ into
a homogeneous multiaffine polynomial of degree $k$:
we will say that the polynomial
\[
   \sum_{I \in {[n] \choose k}} \det(M[I])\boldx^I
 \in \CMA[\boldx]
\]
\defn{represents} $V$, where $\boldx^I := \prod_{i \in I} x_i$.
Not every homogeneous multiaffine polynomial of degree $k$
represents a point of $\Gr(k,n)$.  
A necessary and sufficient condition is
that the coefficients satisfy the quadratic Pl\"ucker relations,
the defining equations for $\Gr(k,n)$ as a projective variety.

If $V \in \Gr(k,n)$ is the column space of a matrix $M$
whose maximal minors are all nonnegative, we say that
$V$ is \defn{totally nonnegative}. 
The \defn{totally nonnegative part} of the Grassmannian, 
denoted $\GrTNN (k,n)$, is the set of all totally nonnegative
$V \in \Gr(k,n)$.

The totally nonnegative part of a flag variety (the Grassmannian
being the most important example)
was first introduced by Lusztig \cite{Lusztig},
as a part of a generalization of the classical theory of totally
nonnegative matrices.
Rietsch showed that totally nonnegative part of any flag
variety has a decomposition into cells
\cite{Rietsch}; Marsh and Rietsch described a parameterization of
the cells \cite{MR}.  In the case of the Grassmannian, Lusztig's
definition agrees with the definition above.  Postnikov
described the indexing of the cells $\GrTNN(k,n)$ and their 
parameterizations in
combinatorially explicit ways \cite{Postnikov},
making $\GrTNN(k,n)$ a very accessible object.
Total nonnegativity has played a key role in a number of recent
applications.
Some of these include: the development of 
cluster algebras \cite{Fomin}; 
soliton solutions to the KP equation \cite{KW}; 
the (remarkably well-behaved) positroid stratification 
of the Grassmannian \cite{KLS}, which has applications to geometric
Schubert calculus \cite{Knutson}.
Our first main result relates total nonnegativity on the Grassmannian 
to stable polynomials.

\begin{theorem}
\label{thm:grassmannian}
Suppose $f(\boldx) \in \CMA[\boldx]$ is a 
homogeneous multiaffine 
polynomial of degree $k$
that represents a point $V \in \Gr(k,n)$.  
Then $f(\boldx)$ is stable if and only
if $V$ is totally nonnegative.
\end{theorem}

The ``phase theorem'' of Choe, Oxley, Sokal, and Wagner
\cite[Theorem 6.1]{COSW}
asserts that if $f(\boldx)
\in \CC[\boldx]$ is stable and homogeneous, then all of its coefficients 
have
the same complex phase, i.e. there is a scalar $\alpha \in \CC^\times$
such that all terms of $\alpha f(\boldx)$ have nonnegative real
coefficients.  The ``only if'' direction of 
Theorem~\ref{thm:grassmannian} is an immediate consequence.
In general, however, the converse of the phase theorem is false: 
for example, $x_1x_2 + x_3x_4$ is not stable.  
Although there are necessary
and sufficient criteria for a polynomial to be stable
(see Theorem~\ref{thm:stabletest}), they 
can be cumbersome to use in practice, and they do not readily 
yield an explicit description 
of the set of stable polynomials as a semialgebraic set.  
It is therefore interesting and
surprising that adding a well-known algebraic condition on the
coefficients (the Pl\"ucker relations) reduces the problem of
testing stability to a simple nonnegativity condition.
This can be seen quite explicitly in the case 
$k=2$, $n=4$; here, the necessary and sufficient 
conditions for stability are tractable, and the Pl\"ucker 
relation trivializes them (see Remark~\ref{rmk:gr24}).

A point $V \in \Gr(k,n)$ determines a representable matroid of 
rank $k$ on the set $[n]$, by taking
the bases to be the indices of the nonzero Pl\"ucker coordinates.
If $V$ is totally nonnegative, this
matroid is called a \defn{positroid}.
The class of positroids is combinatorially well-behaved compared to 
the class of representable matroids.  For example, positroids 
can be enumerated \cite{Williams}.  
Recently, Marcott showed 
that positroids have the \defn{Rayleigh property}
\cite{Marcott}, a property
of matroids closely related to theory stable polynomials.  
This result indicates another relationship between $\GrTNN(n,k)$ and
stable polynomials; it has a similar flavour to
Theorem~\ref{thm:grassmannian},
but neither theorem implies the other.

To prove Theorem~\ref{thm:grassmannian}, 
we establish a second connection between the
theory of stable polynomials and total nonnegativity.
Recall that a matrix $A \in \Mat(n \times n)$ is \defn{totally nonnegative}
if all minors of $A$ are nonnegative.

Let $\Lambda[\boldx]$ denote the complex exterior algebra
generated by $\boldx$, with multiplication denoted $\wedge$, and 
relations 
$ x_i \wedge x_j + x_j \wedge x_i = 0$, for $i, j \in [n]$.
If $I = \{i_1 < i_2 < \dots < i_k\} \subset [n]$, write
$\boldx^{\wedge I} := x_{i_1} \wedge x_{i_2} \wedge \dots \wedge x_{i_k}$.
There is a unique vector space isomorphism
$\xi : \CMA[\boldx] \to \Lambda[\boldx]$
such that $\xi(\boldx^I) = \boldx^{\wedge I}$.
Since $\Lambda[\boldx]$ is a $\Mat(n \times n)$-algebra, 
this isomorphism gives us a linear action of
$\Mat(n \times n)$ on $\CMA[\boldx]$.
Specifically, for $A \in \Mat(n \times n)$,
we have a linear endomorphism
$A_\# : \CMA[\boldx] \to \CMA[\boldx]$,
\[
   A_\# f(\boldx) :=  \xi^{-1} (A \xi(f(\boldx)))\,,
\]
where $A x_j := \sum_{i \geq 0} A_{ij} x_j$, and
$A (x_{j_1} \wedge \dots \wedge x_{j_k}) 
  := Ax_{j_1} \wedge \dots \wedge A x_{j_k}$.
An example of this construction
is given in~\eqref{eqn:2x2}
as part of the the proof of Lemma~\ref{lem:2x2}.

At first glance, the definition of $A_\#$ seems absurd:
we have made a linear identification between part of a 
commutative algebra and a 
supercommutative algebra.  In fact, this issue was already present when
we took Pl\"ucker coordinates as coefficients of a polynomial.
The intuition here is that the difference between these two structures
is in the signs; when we restrict our attention to totally nonnegative 
matrices, or the totally nonnegative part of the Grassmannian, the 
signs are all positive, and the two structures become compatible.

\begin{theorem}
\label{thm:matrix}
For $A \in \Mat(n \times n)$, the following are equivalent:
\begin{packedenum}
\item[(a)]
$A$ is totally nonnegative;
\item[(b)]
for every stable polynomial $f(\boldx) \in \CMA[\boldx]$,
$A_\# f(\boldx)$ is stable.
\end{packedenum}
\end{theorem}

In Section~\ref{sec:stable}, we recall some of the major results from
the theory of stable polynomials.
We then apply this theory to obtain a key lemma, which is roughly
the $n=2$ case of Theorem~\ref{thm:matrix}.
In Section~\ref{sec:positivity}, we discuss some pertinent elements 
of the theory of total nonnegativity and total positivity, for
matrices and for the Grassmannian.  We use these, and our results
from Section~\ref{sec:stable} to prove Theorems~\ref{thm:grassmannian} 
and~\ref{thm:matrix}.  
In Section~\ref{sec:misc} we look at a handful of related results, 
including other families of homogeneous multiaffine stable polynomials,
a family of infinitesimal stability preservers, and
a slightly stronger version of the phase theorem.

\paragraph{Acknowledgements}
The author thanks Cameron Marcott and David Wagner for conversations
that inspired this work.  This research was partially supported by an
NSERC Discovery Grant.

%%%%%%%%%%%%%%%%%%%%%%%%%%%%%%%%%%%%%%%%%%%%%%%%%%

\section{Multiaffine stable polynomials}
\label{sec:stable}

We begin with an example, in which we determine necessary and
sufficient conditions for a degree $2$ homogeneous polynomial 
in $4$ variables to be stable.
This turns out to be the fundamental brute-force calculation 
needed to prove our main theorems.
To obtain such conditions, we use the following criterion for stability
of multiaffine polynomials with real coefficients.  

\begin{theorem}[\Branden \cite{Branden}]%Theorem 5.6
\label{thm:stabletest}
If $f(\boldx) \in \RMA[\boldx]$ is a multiaffine polynomial with
real coefficients, 
define 
\[
  \Delta_{ij} f(\boldx) := 
\tfrac{\partial}{\partial x_i} f(\boldx)  \cdot
\tfrac{\partial}{\partial x_j} f(\boldx) 
-
f(\boldx)  \cdot
\tfrac{\partial^2}{\partial x_i \partial x_j} f(\boldx) 
\,.
\]
Then $f(\boldx)$ is stable if and only if
$\Delta_{ij}f : \RR^n \to \RR$ is a nonnegative function
for all $i,j \in [n]$, $i \neq j$.
\end{theorem}

\begin{example} 
\label{ex:4vars}
Let $a_{12}, a_{13}, a_{14}, a_{23}, a_{24}, a_{34} \geq 0$.
Consider the polynomial
\[
  f(\boldx) = 
  a_{12} x_1x_2 + a_{13}x_1x_3 + a_{14} x_1x_4 + a_{23} x_2x_3
+ a_{24} x_2x_4 + a_{34}x_3x_4\,.
\]
By Theorem~\ref{thm:stabletest}, $f(\boldx)$ is stable iff 
$\Delta_{ij} f \geq 0$ for all $i,j$.
We compute 
\begin{equation}
\label{eqn:delta13}
\Delta_{13} f(\boldx) = 
a_{12} a_{23} x_2^2 + 
(a_{12}a_{34} - a_{13}a_{24} + a_{14}a_{23}) x_2x_4
+ a_{14}a_{34} x_4^2\,.
\end{equation}
Since $a_{ij} \geq 0$, $\Delta_{13} f$ is nonnegative if and only
if its discriminant is nonpositive, i.e.
%: $\Disc \Delta_{12} f \leq 0$.
\begin{equation}
\label{eqn:4vars}
   a_{12}^2 a_{34}^2 
   + a_{13}^2 a_{24}^2 
   + a_{14}^2 a_{23}^2 
   -2 a_{12} a_{34} a_{13} a_{24} 
   -2 a_{13} a_{24} a_{14} a_{23} 
   -2 a_{12} a_{34} a_{14} a_{23} 
   \leq 0 
\,.
\end{equation}
Since this expression is invariant under permutations of $[4]$, we
obtain the same inequality for every other pair of 
indices $i,j \in [4]$.
Hence the inequality \eqref{eqn:4vars} 
is a necessary and sufficient condition for
$f(\boldx)$ be be stable.
\end{example}

\begin{remark} 
\label{rmk:gr24} 
$\Gr(2,4)$ is defined by a single \Plucker relation:
$a_{12}a_{34} - a_{13}a_{24} + a_{14}a_{23} = 0$.  If this holds,
then \eqref{eqn:delta13} is clearly nonnegative, and so
\eqref{eqn:4vars} holds.  This proves the $\Gr(2,4)$ case of
Theorem~\ref{thm:grassmannian}.
However, in general, it is not straightforward
to deduce Theorem~\ref{thm:grassmannian} from
Theorem~\ref{thm:stabletest} using the \Plucker relations.
\end{remark}

A $\CC$-linear map satisfying condition (b) of Theorem~\ref{thm:matrix}
is called a \defn{stability preserver}.  As part of their vast
generalization of the P\'olya--Schur theorem, Borcea and \Branden 
proved that there is an equivalence between stability preservers,
and stable polynomials in twice as many variables.
We state only the multiaffine case of their theorem, as we will
not need the result in its full generality.

\begin{theorem}[Borcea--\Branden \cite{BB1}]%Theorem 1.1
\label{thm:master}
Let $\phi : \CMA[\boldx] \to \CMA[\boldx]$ be a $\CC$-linear map.
Then $\phi$ is a stability preserver if and only if one the following
is true:
\begin{packedenum}
\item[(a)] there is a linear functional $\eta : \CMA[\boldx] \to \CC$
and a stable polynomial $g(\boldx) \in \CMA[\boldx]$ such that
$\phi f(\boldx) = \eta (f(x)) g(\boldx)$;  or
\item[(b)] 
$\phi \big(\prod_{i=1}^n(x_i + y_i)\big) \in 
\CC[\boldx, \boldy]$ is stable.
\end{packedenum}
\end{theorem}

We will refer to stability preservers satisfying (a) as
\defn{rank-one stability preservers}, and those satisfying (b)
as \defn{true stability preservers}.
In (b), we are implicitly extending $\phi$ from a $\CC$-linear map 
$\CMA[\boldx] \to \CMA[\boldx]$ to the unique $\CC[\boldy]$-linear map
$\phi : \CMA[\boldx, \boldy] \to \CMA[\boldx, \boldy]$ that
agrees with the original $\phi$ on $\CC[\boldx]$.
An important property of true stability preservers is that
they are preserved by this natural type of extension.

\begin{proposition}
\label{prop:true}
A linear map $\phi : \CMA[\boldx] \to \CMA[\boldx]$
is a true stability preserver if and only if for any additional set 
of variables $\boldz = (z_1, \dots, z_m)$, the 
$\CC[\boldz]$-linear extension 
$\phi : \CMA[\boldx,\boldz] \to \CMA[\boldx, \boldz]$ 
is a true stability preserver.
\end{proposition}

\begin{proof}
By definition,
$\phi : \CMA[\boldx] \to \CMA[\boldx]$ is a true stability preserver
iff 
\[
  h(\boldx,\boldy) 
  = \phi \big(\textstyle \prod_{i=1}^n(x_i + y_i)\big)
\]
is stable.
The extension $\phi : \CMA[\boldx,\boldz] \to \CMA[\boldx, \boldz]$ 
is a true stability preserver iff
\[
\phi \big(\textstyle \prod_{i=1}^n(x_i+y_i) \cdot
\prod_{j=1}^m (z_j + w_j)\big)
= h(\boldx, \boldy) \prod_{j=1}^m (z_j + w_j)
\] 
is stable.  It is straightforward to verify that
$h(\boldx, \boldy) \in \CC[\boldx, \boldy]$ 
is stable if and only if $h(\boldx, \boldy) \prod_{j=1}^m (z_j + w_j)
 \in \CC[\boldx, \boldy, \boldz, \boldw]$ is stable. The result follows.
\end{proof}

In general, rank-one stability preservers do not have this extendability
property, unless they are also true stability preservers.

%As a consequence of Hurwitz's theorem in complex anlysis,
%\note{[Reference]},
The set of stable polynomials in $\CMA[\boldx]$ is closed.
It follows %from Theorem~\ref{thm:master}, 
that the set of
true stability preservers $\CMA[\boldx] \to \CMA[\boldx]$,
being linearly equivalent 
to the set of stable polynomials in $\CMA[\boldx, \boldy]$,
is also closed.  These facts will be used in the next section.

We conclude this section by using Theorem~\ref{thm:master} to
prove the following lemma, which is almost-but-not-quite the $n=2$
case of Theorem~\ref{thm:matrix}.

\begin{lemma}
\label{lem:2x2}
If $Q \in \Mat(2 \times 2)$ is totally nonnegative, then
$Q_\# : \CMA[x_1,x_2] \to \CMA[x_1,x_2]$ is a true stability preserver.
\end{lemma}

\begin{proof}
Write $Q = \left( \begin{smallmatrix} a & c \\ b & d\end{smallmatrix}\right)$.
Then we have
\begin{equation}
\label{eqn:2x2}
\begin{aligned}
Q_\#(1) &= 1 
\\
Q_\#(x_1) &= a x_1 + bx_2
\\
Q_\#(x_2) &= c x_1 + dx_2
\\
Q_\#(x_1x_2) &= (ad-bc) x_1x_2\,.
\end{aligned}
\end{equation}
Thus $Q_\#$ is a true stability preserver if and only if
\[
   h(\boldx, \boldy) 
   = y_1y_2 + a x_1y_2 + b x_2y_2 + cx_1y_1 + d x_2y_1 + (ad-bc)x_1x_2
\]
is stable.  

Now assume $Q$ is totally nonnegative.  Then all coefficients
coefficients of $h(\boldx, \boldy)$ are nonnegative.  
As we saw in Example~\ref{ex:4vars}, $h(\boldx,\boldy)$ is stable if 
and only
if the inequality \eqref{eqn:4vars} holds, which in this case 
amounts to
\[
    (ad-bc)^2 + a^2d^2 + b^2c^2 - 2ad(ad-bc) -2bc(ad-bc) -2adbc \leq 0
\,,
\]
or equivalently
\[
   -4bc(ad-bc) \leq 0\,.
\]
Since $b \geq 0$, $c \geq 0$, and $ad-bc \geq 0$,
the result follows.
\end{proof}

%%%%%%%%%%%%%%%%%%%%%%%%%%%%%%%%%%%%%%%%%%%%%%%%%%

\section{Total positivity}
\label{sec:positivity}

A matrix $A \in \Mat(n\times n)$ is \defn{totally positive} if all of its
minors are strictly positive.  We denote the set of totally positive
$n \times n$ matrices by $\MatTP(n \times n)$, and we denote the
set of totally nonnegative matrices 
by $\MatTNN(n \times n)$.
Lying between these is the set 
$\GLTNN(n) = \MatTNN(n \times n) \cap \GL(n)$ 
of invertible totally nonnegative matrices.  Each of the
sets $\MatTNN(n \times n)$, $\MatTP(n \times n)$ and $\GLTNN(n)$
is a multiplicative semigroup, i.e. closed under matrix 
multiplication.  We have containments
\[
  \MatTP(n \times n) \subset \GLTNN(n) \subset \MatTNN(n \times n)
\]
and
$\MatTNN(n \times n)$ is the closure of all of these sets
\cite{Whitney}.

The Loewner--Whitney theorem \cite{Loewner}
describes the generators of $\GLTNN(n)$.
Let
\[
D_i(t) :=
\begin{pmatrix}
1 & 0 & \cdots  & 0 &  \cdots  & 0 & 0  \\
0 & 1 & \cdots   & 0 & \cdots   & 0 & 0  \\
\vdots  & \vdots  & \ddots  &   &          &\vdots   & \vdots \\
0 & 0 &         & t &          & 0 & 0  \\
\vdots  &\vdots   &         &   &  \ddots  &\vdots   & \vdots \\
0 & 0 & \cdots  & 0 &  \cdots  & 1 & 0  \\
0 & 0 & \cdots  & 0 &   \cdots & 0 & 1  \\
\end{pmatrix}
\qquad
E_i(t) :=
\begin{pmatrix}
1 & 0 & \cdots  & 0 & 0 &  \cdots  & 0 & 0  \\
0 & 1 & \cdots   & 0 & 0 & \cdots   & 0 & 0  \\
\vdots  & \vdots  & \ddots  & &   &          &\vdots   & \vdots \\
0 & 0 &         & 1 & t &          & 0 & 0  \\
0 & 0 &         & 0 & 1 &          & 0 & 0  \\
\vdots  & \vdots  &   & &   &   \ddots       &\vdots   & \vdots \\
0 & 0 & \cdots  & 0 & 0 &  \cdots  & 1 & 0  \\
0 & 0 & \cdots  & 0 & 0 &   \cdots & 0 & 1  \\
\end{pmatrix}
\]
where in each case, $t$ appears in row $i$; let $F_i(t)$ be
the transpose of $E_i(t)$.
$\GLTNN(n)$ is the semigroup generated by all
$D_i(t)$, $E_i(t)$, $F_i(t)$, $t > 0$.
We use this description to prove Theorem~\ref{thm:matrix}.

\begin{proof}[Proof of Theorem~\ref{thm:matrix}]
We begin with the implication (a) $\Rightarrow$ (b).  
We will show that if $A \in \GLTNN(n)$ then
$A_\#$ is a true stability preserver. 
Since
$\MatTNN(n \times n)$ is the closure of $\GLTNN(n)$,
and the set of true stability preservers 
%$\CMA[\boldx] \to \CMA[\boldx]$
is closed,
this is implies the result for $A \in \MatTNN(n \times n)$.

Since $\GLTNN(n)$ is a semigroup, and $(AB)_\# = A_\# B_\#$
for $A, B \in \Mat(n \times n)$, it suffices to prove this in
the case where $A$ is a generator for $\GLTNN(n \times n)$,
i.e. one of $D_i(t)$, $E_i(t)$, $F_i(t)$,  $t > 0$.  In each case,
we can write
\[
  A = 
  \begin{pmatrix}
   I_k & 0 & 0 \\
   0 & Q & 0 \\
   0 & 0 & I_{n-k-2} \\
\end{pmatrix}
\]
where $0 \leq k \leq n-2$, and 
$Q = \left( \begin{smallmatrix} a & c \\ b & d\end{smallmatrix}\right)$
is some totally nonnegative $2 \times 2$
matrix.  Observe that 
\begin{equation}
\label{eqn:matrix}
   A_\# \boldx^I =
\begin{cases}
 \boldx^{J} 
    &\quad \text{if $k+1 \notin I$, $k+2 \notin I$}
\\
(ax_{k+1}+bx_{k+2}) \boldx^{J}
    &\quad \text{if $k+1 \in I$, $k+2 \notin I$}
\\
(cx_{k+1}+dx_{k+2}) \boldx^{J}
    &\quad \text{if $k+1 \notin I$, $k+2 \in I$}
\\
(ad-bc)x_{k+1}x_{k+2} \boldx^{J}
    &\quad \text{if $k+1 \in I$, $k+2 \in I$}\,.
\end{cases}
\end{equation}
where $J = I \setminus \{k+1, k+2\}$.  Comparing \eqref{eqn:matrix} 
with \eqref{eqn:2x2}, we see that
$A_\#$ is the unique $\CC[x_1, \dots, x_k, x_{k+3}, \dots, x_n]$-linear
extension of $Q_\# : \CMA[x_{k+1}, x_{k+2}] \to
\CMA[x_{k+1}, x_{k+2}]$.
By
Lemma~\ref{lem:2x2},
$Q_\#$ is a true stability preserver, and therefore 
by Proposition~\ref{prop:true} so is $A_\#$.

For the implication (b) $\Rightarrow$ (a), 
suppose that $A_\#$ is a stability preserver.
If $A$ is the zero matrix, then $A$ is certainly totally nonnegative.
Otherwise, $A_\#$ has rank at least $2$,
so by Theorem~\ref{thm:master}, it must be a true stability preserver,
i.e.
\[
  h(\boldx, \boldy) = A_\# \big( \prod_{i=1}^n(x_i+ y_i)\big)
\]
is stable.  Since $A_\#$ preserves degree, $h(\boldx, \boldy)$ 
is homogeneous of degree $n$, and since $A_\#$ acts trivially on constants,
the coefficient of $\boldy^{[n]}$
in $h(\boldx, \boldy)$ is $1$. Therefore by the phase theorem
all coefficients of $h(\boldx, \boldy)$ must be nonnegative.
More generally the coefficient of $\boldx^I\boldy^J$ in 
$h(\boldx, \boldy)$ is
the minor of $A$ corresponding to row set $[n] \setminus J$, 
and column set $I$.  Since all minors of $A$ are coefficients
of $h(\boldx, \boldy)$, we deduce that all minors of $A$ are nonnegative.
\end{proof}

The \defn{totally positive part} of the Grassmannian $\Gr(k,n)$,
denoted $\GrTP(k,n)$ is
the set of $V \in \Gr(k,n)$ such that all \Plucker coordinates
of $V$ are strictly positive.
Since totally positive matrices are invertible, they act on the
Grassmannian $\Gr(k,n)$, and the totally positive part of the
Grassmannian is an ``orbit''.
Specifically, let $V_0 \in \GrTNN(k,n)$ be the column space of $M_0 = 
\left(\begin{smallmatrix} I_k \\ 0 \end{smallmatrix}\right)$.
Then we have $\GrTP(k,n) = \{AV_0 \mid A \in \MatTP(n \times n)\}$,
where $AV_0$ is defined to be the column space of the matrix $AM_0$.
The totally nonnegative part of the Grassmannian $\GrTNN(k,n)$
does not have such a straightforward 
relationship to $\MatTNN(n \times n)$,
but is the closure of $\GrTP(k,n)$.  These facts are essentially
Lusztig's definitions of $\GrTP(k,n)$ and $\GrTNN(k,n)$ 
\cite{Lusztig}.

\begin{proof}[Proof of Theorem~\ref{thm:grassmannian}]
As already noted in the introduction, if $f(\boldx)$ is a stable 
polynomial representing
$V \in \Gr(k,n)$, then by the phase theorem, $V \in \GrTNN(k,n)$.
It remains to prove that if $f(\boldx)$ represents a 
point $V \in \GrTNN(k,n)$, then $f(\boldx)$ is stable.

Since $\GrTNN(k,n)$ is the closure 
of $\GrTP(k,n)$, and since
the set of multiaffine stable polynomials is closed, it suffices
to prove the theorem when %$f(\boldx)$ represents a point 
$V \in \GrTP(k,n)$.  If this is the case,
there exists an totally positive matrix $A \in \MatTP(n \times n)$
such that $V = A V_0$.  Note that the monomial $\boldx^{[k]}$ 
represents $V_0 \in \GrTNN(k,n)$.
Since the action of $A_\#$ on multiaffine
polynomials is defined via an isomorphism with the exterior algebra, 
we have that $f(\boldx) = A_\# \boldx^{[k]}$.  By Theorem~\ref{thm:matrix},
$A_\#$ is a stability preserver, and $\boldx^{[k]}$ is stable,
so $f(\boldx)$ is stable.
\end{proof}

%%%%%%%%%%%%%%%%%%%%%%%%%%%%%%%%%%%%%%%%%%%%%%%%%%

\section{Odds and ends}
\label{sec:misc}

There is a second connection between Theorems~\ref{thm:grassmannian}
and~\ref{thm:matrix}.  If $A \in \Mat(n \times n)$, let
$A^\vee \in \Mat(n\times n)$ denote the matrix
   $A^\vee_{i,j} = (-1)^{n-j}A_{n+1-i,j}$.
Let $V \in \Gr(n,2n)$ be the column space of the $2n \times n$ matrix
$\left(\begin{smallmatrix} I_n \\ A^{\vee} \end{smallmatrix}\right)$.
It is not hard to check the following facts:
\begin{packeditemize}
\item
$V \in \GrTNN(n,2n)$ if and only if $A \in \MatTNN(n \times n)$.
\item
$A_\#(\prod_{i=1}^n (x_i+y_i))$ represents $V$, with the variables
ordered $y_1 < y_2 < \dots < y_n < x_n <  \dots < x_2 < x_1$.
\end{packeditemize}
Thus we see that Theorem~\ref{thm:grassmannian} implies
Theorem~\ref{thm:matrix}, though not by reversing the argument in
Section \ref{sec:positivity}:
$A_\#$ is a stability preserver iff
$A_\#(\prod_{i=1}^n (x_i+y_i))$ is stable iff
$V \in \GrTNN(n,2n)$ iff $A \in \MatTNN(n \times n)$.

There is another class of stable polynomials comes that from
the minors of a matrix.  If $M \in \Mat(n \times k)$, then the polynomial
\begin{equation}
\label{eqn:pluckersquares}
    \sum_{I \in {[n] \choose k}} \big|\det(M[I])\big|^2 \boldx^I
 \in \CMA[\boldx]
\end{equation}
is always stable \cite[Theorem 8.1]{COSW}.  
This raises the question: to what extent do these classes overlap?

The answer is not much.
For dimensional reasons, a general polynomial of the form
\eqref{eqn:pluckersquares} does not represent a point of $\GrTNN(k,n)$.
On the other hand, with the exception of a few small cases,
a point of $\GrTP(k,n)$ cannot be represented by
a polynomial of the form \eqref{eqn:pluckersquares}.
For ease of notation, we present the argument for $\Gr(2,6)$,
though the same idea works for $k \geq 2$, $n-k \geq 4$.

\begin{proposition}
No point of $\GrTP(2,6)$ is represented by a
polynomial of the form~\eqref{eqn:pluckersquares}.
\end{proposition}

\begin{proof}
Suppose to the contrary that $\sum a_I \boldx^I$ 
represents a point of $\GrTP(2,6)$, and $a_I = |b_I|^2$
where $b_I = \det(M[I])$ for some matrix $M$.  Then both
$\big[a_I : I \in {[6] \choose 2}\big]$ and
$\big[b_I : I \in {[6] \choose 2}\big]$ satisfy the \Plucker
relations.  For $\Gr(2,6)$ these are:
\begin{align*}
   a_{ik}a_{jl} &= a_{ij}a_{kl} + a_{il}a_{jk} \\
   b_{ik}b_{jl} &= b_{ij}b_{kl} + b_{il}b_{jk} 
\end{align*}
for $1 \leq i < j < k < l \leq 6$.  Multiplying
the second equation by its complex conjugate and using $a_I = |b_I|^2$, 
we find that
$b_{ij}b_{kl}\overline{b_{il}b_{jk}}$ 
and $\overline{b_{ij}b_{kl}}b_{il}b_{jk}$ are pure imaginary.
In particular,
\[
b_{12} b_{35} \overline{b_{15}b_{23}}
\qquad
\overline{b_{12} b_{36}} b_{16}b_{23}
\qquad
\overline{b_{34} b_{56}} b_{36}b_{45}
\qquad
\overline{b_{13} b_{45}} b_{15}b_{34}
\qquad
b_{13} b_{56} \overline{b_{16}b_{35}}
\]
are all pure imaginary.  The product of these five pure 
imaginary numbers must be pure imaginary. But instead,
their product is 
$a_{12}a_{13}a_{15}a_{16}a_{23}a_{34}a_{35}a_{36}a_{45}a_{56} > 0$.
This is a contradiction.
\end{proof}

A related result replaces the determinant of $M[I]$ with the 
permanent.
If $M \in \Mat(n \times k)$ is a matrix with nonnegative real entries
then
\begin{equation}
\label{eqn:permanent}
    \sum_{I \in {[n] \choose k}} \mathrm{per}(M[I]) \boldx^I
 \in \CMA[\boldx]
\end{equation}
is a stable polynomial \cite[Theorem 10.2]{COSW}.
It would be surprising if it were typically possible to represent
a point of $\GrTP(k,n)$ by a polynomial of the 
form~\eqref{eqn:permanent}.  For example, 
it is not hard to show this impossible if
$k=2$, $n \geq 5$,
but at present we do not have a general proof.

A multiaffine polynomial $f(\boldx) \in \RMA[\boldx]$ is said to be
a \defn{Rayleigh polynomial} if $\Delta_{ij}f : \RR_{\geq 0}^n \to \RR$
is nonnegative for all $i, j \geq 0$.  This is a relaxation of
than the criterion for stability in Theorem~\ref{thm:stabletest}:
for the not-so-keenly observant, the %difference is that the
Rayleigh condition only requires $\Delta_{ij}f \geq 0$ 
on nonnegative inputs, whereas stability requires 
$\Delta_{ij} f \geq 0$ on all real inputs.
Thus real multiaffine stable polynomials are Rayleigh, but in
general the converse is not true.  
Given $V \in \Gr(k,n)$, 
let $\calB \subset {[n] \choose k}$ be the set of indices of
the nonzero \Plucker coordinates of $V$.
$\calB$ is the (set of bases of) a representable matroid, and
if $V \in \GrTNN(k,n)$, $\calB$ is called a positroid.
Marcott has recently proved the following result.

\begin{theorem}[Marcott \cite{Marcott}]
\label{thm:positroid}
If $\calB$ is a positroid, then 
$B(\boldx) = \sum_{I \in \calB} x^I$ is a Rayleigh polynomial.
\end{theorem}

This is much like the harder direction of
Theorem~\ref{thm:grassmannian}, except that the
coefficients in the polynomial have been stripped away.  
The converse of 
Theorem~\ref{thm:positroid} is not true: if $\calB$ is matroid that is
not a positroid, 
then $B(\boldx)$ may or may not be Rayleigh; there is no known
classification of Rayleigh matroids.  It is also not presently known 
whether, for positroids, $B(\boldx)$ is a stable polynomial.

\bigskip

We mention two applications of the ideas developed in this paper.
A linear endomorphism 
$\delta : \CMA[\boldx] \to \CMA[\boldx]$ is an
\defn{infinitesimal stability preserver} if $\exp(t\delta) : \CMA[\boldx]
\to \CMA[\boldx]$ is a stability preserver for all $t \geq 0$.
The set of all infinitesimal stability preservers is a closed convex cone
in the space of all operators on $\CMA[\boldx]$.
This can be seen as follows:
if $\alpha$ and $\beta$ are infinitesimal stability
preservers, then for $t \geq 0$, $\exp(t(\alpha+\beta)) 
= \lim_{m \to \infty} 
\big(\exp(\frac{t}{m}\alpha) \exp(\frac{t}{m}\beta)\big)^m$
is a stability preserver, and hence $\alpha+\beta$ is an infinitesimal
stability preserver.

Our first application is an example of a non-trivial
family of infinitesimal stability preservers.
Let $Z \in \Mat(n \times n)$ be a matrix with real diagonal entries, 
and nonnegative off-diagonal entries.  Define $\delta_Z : \CMA[\boldx] \to
\CMA[\boldx]$ by 
\[
  \delta_Z x^J = 
           \sum_{j \in J} \left(Z_{jj} +
           \sum_{i \in [n] \setminus J} 
            Z_{ij} \frac{x_i}{x_j} \right) x^J
\]
for all $J \subset [n]$, and extending linearly. 

\begin{proposition}
\label{prop:inf-tnn}
If $Z$ is tridiagonal (i.e. $Z_{ij} = 0$ for $|i-j| >1$), then 
for all $t \geq 0$, $\exp(tZ)$ is totally nonnegative,
and $\exp(t \delta_Z) = \exp (tZ)_\#$.
\end{proposition}

\begin{proof}
The fact that $\exp(tZ)$ is totally nonnegative follows from
the Loewner--Whitney theorem, one formulation of which is that 
matrices of this form infinitesimally generate $\GLTNN(n)$.  To see
that $\exp(t \delta_Z) = \exp (tZ)_\#$, we need to verify that
$\delta_Z = \frac{\partial}{\partial t} \exp(tZ)_\# \big|_{t = 0}$.
But since $Z \mapsto \delta_Z$, and $Z \mapsto 
\frac{\partial}{\partial t} \exp(tZ)_\# \big|_{t = 0}$ are both
linear maps, it suffices to check this when $Z$ has
a single nonzero entry; this is straightforward.
\end{proof}

\begin{proposition}
\label{prop:inf-preserver}
For any $Z \in \Mat(n \times n)$ with real diagonal entries, and 
nonnegative off-diagonal
entries, $\delta_Z$ is an infinitesimal stability preserver.
\end{proposition}

\begin{proof}
First suppose $Z$ is tridiagonal.
In this case, by Proposition~\ref{prop:inf-tnn} and Theorem~\ref{thm:matrix}, 
we have that $\exp(t \delta_Z) = \exp(t Z)_\#$ is a stability
preserver for all $t \geq 0$; hence $\delta_Z$ is an infinitesimal
stability preserver.

Next suppose that
$Z = Q_1Z_1Q_1^{-1}$ for some permutation matrix $Q_1$ and some 
tridiagonal matrix $Z_1$.
Since the definition of $\delta_Z$ is symmetric in variables 
$x_1, \dots, x_n$, it is clear that $\delta_Z$ is an infinitesimal
stability preserver in this case too.

Finally observe that a general $Z$ can be written as
\[
  Z = \sum_{i=1}^s Q_i Z_i Q_i^{-1}
\]
where each $Z_i$ is a real tridiagonal matrix with 
nonnegative off-diagonal entries, and $Q_i$ is a permutation matrix.
Since the map $Z \mapsto \delta_Z$ is linear, we see that
$\delta_Z$ is a sum of infinitesimal stability preservers,
and the result follows.
\end{proof}

\begin{remark}
The proof of Proposition~\ref{prop:inf-preserver}
is fundamentally the same as the proof of
\cite[Proposition 5.1]{BBL}, which also establishes a family of
infinitesimal stability preservers.
The two families are superficially similar but neither is a special
case of the other.
Concretely, the operators in \cite{BBL} are given by 
$x^J \mapsto \sum_{j \in J} 
           \sum_{i \in [n] \setminus J} 
            Z_{ij} \left(\frac{x_i}{x_j} - 1 \right) x^J$
for a real symmetric matrix $Z$; the exponentials of the operators in this
family are doubly stochastic, and have the physical interpretation
as generators for a symmetric exclusion process on $n$ sites.
Proposition~\ref{prop:inf-preserver} seems to be about the best one
can do to mimic this construction for asymmetrical matrices.
\end{remark}

As a second application, we prove a slightly
more general version of the phase theorem.

\begin{theorem}
\label{thm:phaseplus}
Let $f(\boldx) \in \CC[\boldx]$ be a stable polynomial.
If $f(\boldx)$ has no terms of degree $k$, $k \in \ZZ$,
then there exists a nonzero scalar $\alpha \in \CC^\times$ such that
of all terms of degree $k+1$ in $\alpha f(x)$ 
and all terms of degree $k-1$ in $-\alpha f(x)$ 
have nonnegative real coefficients.
\end{theorem}

\begin{remark} 
There cannot be large gaps in the degrees of a stable polynomial:
if $f(\boldx)$ is stable and has no terms of degree $k$, then either
$f(\boldx)$ has terms of \emph{both} degree $k+1$ and $k-1$, or
$k > \mathop\mathrm{maxdeg} f(\boldx)$, or
$k < \mathop\mathrm{mindeg} f(\boldx)$.  This can be deduced from
the corresponding fact for single variable polynomials, or
from an argument similar to the one presented below.
It follows that Theorem~\ref{thm:phaseplus} also implies the stronger 
version of the phase theorem in \cite[Theorem 6.2]{COSW}.
\end{remark}

The \defn{support}
of a polynomial $f(\boldx)$ is the set of monomials in $\CC[\boldx]$
that appear in $f(\boldx)$ with a nonzero coefficient.
Define $\|f(\boldx)\|$ to be the maximum of the absolute values
of the coefficients of $f(\boldx)$.
For example, if $f(\boldx) = 4x_1x_2^2 - x_1^3$, then the support
of $f(\boldx)$ is $\{x_1x_2^2, x_1^3\}$, and $\|f(\boldx)\| = 4$.

\begin{lemma}
Let $f(\boldx) \in \CMA[\boldx]$ be a multiaffine stable polynomial.
For every $\varepsilon > 0$, there exists a polynomial 
$f_\varepsilon(\boldx) \in \CMA[\boldx]$ such that
\begin{packedenum}
\item[(i)] $f_\varepsilon(\boldx)$ is stable;
\item[(ii)] $\|f(\boldx) - f_\varepsilon(\boldx)\| < \varepsilon$; 
\item[(iii)] for all $k \in \ZZ$, 
if $f(\boldx)$ has no terms of
degree $k$ then $f_\varepsilon(\boldx)$ has no terms of degree $k$; and
\item[(iv)] if $f(\boldx)$ has a term of degree $k$, then
the support of $f_\varepsilon(\boldx)$
contains all multiaffine monomials of degree $k$.
\end{packedenum}
\end{lemma}

\begin{proof}
Take $f_\varepsilon(\boldx) \in \CMA[\boldx]$ such that
(i)--(iii) above are satisfied, and
subject to these conditions $f_\varepsilon(\boldx)$ has maximal support.
We claim that $f_\varepsilon(\boldx)$ must also satisfy property (iv).
If not, then
we can find a matrix $A = E_i(t)$ or $F_i(t)$ such that
for all but finitely many $t \in \RR$,
$A_\# f_\varepsilon(\boldx)$ has strictly larger support 
than $f_\varepsilon(\boldx)$.
By taking $t>0$ sufficiently small, 
we can achieve 
$\|f_\varepsilon(\boldx) - A_\# f_\varepsilon(\boldx)\| 
< \varepsilon - \|f(\boldx) - f_\varepsilon(\boldx)\|$.
Thus, $\|f(\boldx) - A_\# f_\varepsilon(\boldx)\| < \varepsilon$,
i.e. $A_\# f_\varepsilon(\boldx)$ satisfies (ii).
By Theorem~\ref{thm:matrix}, and $A_\# f_\varepsilon(\boldx)$ satisfies
(i), and since $A_\#$ 
preserves degree, $A_\# f_\varepsilon(\boldx)$ satisfies
(iii).
Thus we have a contradiction in the choice of $f_\varepsilon(\boldx)$.
\end{proof}

\begin{proof}[Proof of Theorem~\ref{thm:phaseplus}]
By the Grace--Walsh--Szeg\"o coincidence theorem 
\cite{Grace, Szego, Walsh} it suffices to prove this in the
case where $f(\boldx) \in \CMA[\boldx]$ is a multiaffine 
polynomial.  

Consider $f_\varepsilon(\boldx)$, $\varepsilon > 0$.
For any $I \subset [n]$, and any $i, j \in [n] \setminus I$
we can write 
\[
   f_\varepsilon(\boldx) = \boldx^I(a + bx_i + cx_j + dx_ix_j) 
  + \ldots
\]
where the $\ldots$ indicates all terms that are not of this form.  
It is straightforward 
(using Theorem~\ref{thm:master} or elementary arguments)
to check that
the linear map $\phi : \CC[\boldx] \to \CC[\boldx]$ defined
by 
\[
   \phi(x^J) = 
\begin{cases}
x^{J \setminus I} & \quad \text{if $I \subset J \subset I \cup \{i,j\}$} \\
0 & \quad\text{otherwise}
\end{cases}
\]
is a stability preserver.  Thus $\phi(f_\varepsilon(\boldx)) = 
a+ bx_i +c x_j + dx_ix_j$
is stable and multiaffine in two variables.  
If $|I| = k$, then $a = 0$
from which it is easy to show that $b$ and $c$ have same phase.
By property (iv) of $f_\varepsilon(\boldx)$, $b\neq0$ iff $c\neq0$, 
which implies that the ``same phase'' relation is transitive.
Thus by considering all $I$ with $|I| = k$, we see that 
all terms of degree $k+1$ in $f_\varepsilon(\boldx)$ have 
the same phase.
Similarly $|I| = k-2$, then $d=0$ and we have the same 
result for terms of degree $k-1$.  
If $|I| = k-1$ then $b=c=0$, and we deduce that $d$ and $-a$ have the 
same phase.  
This shows that the result is true for the polynomial 
$f_\varepsilon(\boldx)$.  
The theorem now follows, since 
$f(\boldx) = \lim_{\varepsilon \to 0} f_\varepsilon(\boldx)$.
\end{proof}

%%%%%%%%%%%%%%%%%%%%%%%%%%%%%%%%%%%%%%%%%%%%%%%%%%

\bigskip

\noindent
\footnotesize%
%Address
   \textsc{Combinatorics and Optimization Department,
       University of Waterloo, 200 University Ave. W.  Waterloo,
       ON, N2L 3G1, Canada.} \texttt{kpurbhoo@uwaterloo.ca}.

\end{document}